\documentclass[11pt]{amsart}

\usepackage[nocompress]{cite}
\usepackage{amsmath,amsfonts,amssymb,amsthm}

\usepackage{hyperref}
\usepackage{cleveref}
\usepackage{geometry}
\usepackage{enumitem}
\hypersetup{
  colorlinks=true,
  linkcolor=blue,
  citecolor=blue,
  urlcolor=blue
}

\newtheorem{theorem}{Theorem}[section]
\newtheorem{proposition}[theorem]{Proposition}
\newtheorem{lemma}[theorem]{Lemma}
\newtheorem{corollary}[theorem]{Corollary}

\theoremstyle{definition}
\newtheorem{definition}[theorem]{Definition}

\theoremstyle{remark}
\newtheorem*{remark}{Remark}

\numberwithin{equation}{section}

\title[Nodal set comparison for Allen--Cahn solutions with conical asymptotics]
      {Nodal set comparison for Allen--Cahn solutions with conical asymptotics}

\author{Sanghoon Lee}
\address{Department of Mathematics  \\  The Chinese University of Hong Kong \\ Shatin, N.T., Hong Kong}
\email{sanghoonlee@cuhk.edu.hk}
\thanks{}

\author{Taehun Lee}
\address{Department of Mathematics, Konkuk University, Seoul 05029, Korea}
\email{taehun@konkuk.ac.kr}
\thanks{T. Lee was supported by the NRF grant funded by the Korea government (MSIT) (RS-2023-00211258).}

\begin{document}
\subjclass[2020]{Primary 35J61; Secondary 35B08, 35B50, 53A10}
\keywords{Allen--Cahn equation, minimal cone, nodal set, comparison principle}
\maketitle

\begin{abstract}

We establish a comparison principle for entire solutions of the Allen–Cahn equation whose nodal sets, possibly singular, are asymptotic to a regular minimizing hypercone. We show that inclusion of the positive phases enforces a global ordering of the solutions. As a consequence, the positive phase uniquely determines the solution, and strict phase inclusion implies that the corresponding nodal sets are disjoint. Our analysis relies on a maximum principle for the linearized operator on unbounded domains that are not necessarily smooth, and yields an Allen–Cahn analogue of the strong maximum principle for minimal hypersurfaces.

\end{abstract}
\section{Introduction}

This paper concerns an ordering property and rigidity results for entire solutions to the Allen--Cahn equation
\begin{align}\label{eq:main}
\Delta u + u-u^3 = 0 \quad \text{in } \mathbb{R}^{n},
\end{align}
specifically focusing on solutions whose nodal sets are asymptotic to a regular minimizing hypercone $\mathcal{C}$ (that is, a minimizing hypercone with the origin as its only singularity).
Equation \eqref{eq:main} is a classical model arising in the theory of phase transitions, describing the separation of a binary mixture into two stable phases represented by $u = \pm 1$.

A central theme in the analysis of \eqref{eq:main} is the deep connection between the nodal sets (i.e., zero level sets) of entire solutions and minimal hypersurfaces.
This connection was famously highlighted by De Giorgi’s conjecture \cite{DeG79}, which asserts that any bounded solution that is monotone in one direction must be one-dimensional—equivalently, must have planar level sets—in dimensions $n \le 8$.
The conjecture was affirmatively resolved by Ghoussoub and Gui \cite{GG98} for $n = 2$, and by Ambrosio and Cabré \cite{AC00} for $n = 3$. It was later established by Savin \cite{Savin09} for $4 \le n \le 8$ under a mild additional assumption. An alternative proof was subsequently given by Wang \cite{Keleiwang}.
In contrast, in dimensions $n \ge 9$, del Pino, Kowalczyk, and Wei \cite{dPKW11} constructed a counterexample, demonstrating that the dimensional restriction $n \le 8$ in De Giorgi’s conjecture is sharp. This result may be viewed as the Allen–Cahn analogue of the celebrated counterexample of Bombieri, De Giorgi, and Giusti \cite{BDGG69}.

There has been extensive study over the past decades on the structure of nodal sets and the moduli space of solutions to the Allen–Cahn equation; see, for example, \cite{HLS+21, MR2227143, MR3945835, MR3743704, MR1803974, MR4831349, MR2434903, MR2948876, MR2995371, MR3019512, MR3935478, MR4021161, MR4050103, chan2025globalstablesolutionsfree}, to list only a few.
More recently, significant progress toward the stable version of De Giorgi’s conjecture in dimension $n=4$ was made by Florit–Simon and Serra \cite{floritsimon2025stablesolutionsallencahnequation}.
Nevertheless, our understanding in this direction remains substantially less developed than in the theory of stable minimal hypersurfaces, most notably in connection with the stable Bernstein problem, both in low and high dimensions.

In dimensions $n \ge 8$, a variety of examples are known whose nodal sets are non-planar. In particular, solutions whose nodal sets are given precisely by Simons cone $\mathcal{C}_{m,m}$,
\begin{align}\label{eq:simonscone} \mathcal{C}_{m,m} = \big\{ (x, y) \in \mathbb{R}^m \times \mathbb{R}^m : |x| = |y| \big\}, 
\end{align}
have been constructed and studied by Cabré and Terra \cite{CT09, CT10, Cabre12}. These are known as saddle-shaped solutions (see \Cref{def:saddle}). Moreover, motivated by the geometric picture provided by the Hardt–Simon foliation \cite{HS85}—a foliation of $\mathbb{R}^{2m} \setminus \mathcal{C}_{m,m}$ by a one-parameter family of smooth minimal hypersurfaces asymptotic to $\mathcal{C}_{m,m}$—  Liu, Wang, and Wei \cite{LWW17} constructed entire solutions whose nodal sets are asymptotic to these leaves and proved that they are global minimizers, using the family of solutions constructed by Pacard and Wei \cite{PW13} as barriers.
More recently, Agudelo and Rizzi \cite{AR25_arxiv} established analogous existence results for solutions whose nodal sets are asymptotic to Lawson cones.

Inspired by the uniqueness of the Hardt–Simon foliation, it is natural to ask whether the nodal sets of the solutions constructed above themselves foliate Euclidean space, and whether these solutions exhaust all entire solutions whose nodal sets are asymptotic to Simons cone (see \cite{Solomon, Nick} in the setting of minimal hypersurfaces).
As a necessary step toward these goals, we establish the following strong maximum principle for nodal sets.

\begin{theorem}\label{thm:main-order}
Let $\mathcal{C} \subset \mathbb{R}^n$ be a regular minimizing hypercone. Suppose that $u_1$ and $u_2$ are nonconstant solutions of \eqref{eq:main} whose nodal sets are asymptotic to $\mathcal{C}$ in the $C^1$ sense (see Definition~\ref{def:asymptotic_C1}).
If $\{u_1 > 0\} \subseteq \{u_2 > 0\}$, then $u_1 \le u_2$ in $\mathbb{R}^n$. Moreover, either $u_1 \equiv u_2$ or $u_1 < u_2$ everywhere.
\end{theorem}

Strict inclusion of the positive phases forces the corresponding solutions to be strictly ordered, thereby ensuring that their nodal sets never touch. This result may be viewed as an Allen–Cahn analogue of the strong maximum principle for minimal hypersurfaces \cite{Simon87, SW89, Ilmanen96, Wickramasekera14}. However, it is important to note that a purely local version of such a strong maximum principle for nodal sets is not available in general.
Consequently, \Cref{thm:main-order} implies that the nodal set uniquely determines the solution. This generalizes the uniqueness theorem of Cabré \cite{Cabre12} to the setting in which the nodal set is asymptotic to a cone rather than coinciding exactly with it, and also allows for the presence of singularities on the nodal set.

\begin{corollary}[Uniqueness of solutions]\label{cor:phase-equality}
Let $u_1$ and $u_2$ be solutions satisfying the assumptions of \Cref{thm:main-order}. If their positive phases coincide, that is, $\{u_1 > 0\} = \{u_2 > 0\}$, then $u_1 \equiv u_2$ in $\mathbb{R}^n$.
\end{corollary}

Our proof extends the framework of the maximum principle for the linearized operator on unbounded domains developed by Cabré \cite{Cabre12}.
A key distinction is that we work with domains whose boundaries are not necessarily smooth and may exhibit more general types of singularities beyond the conical case.
This added complexity necessitates a generalization of several fundamental analytic tools—including the maximum principle, the analysis of asymptotic behavior, and the existence theory for minimal solutions—to this broader asymptotic setting. In particular, we carry out a refined analysis near the singular part of the nodal set, namely $\{ u = 0 \} \cap \{ |\nabla u | = 0 \}$.

We remark that the condition of being asymptotic to $\mathcal{C}$ in the $C^1$ sense is satisfied by a broad class of solutions in the literature. This includes not only the saddle-shaped solutions whose nodal set coincides with the cone itself, as studied in \cite{CT09,CT10}, but also the solutions whose nodal sets correspond to the Hardt-Simon leaves associated with the cone, such as those constructed in \cite{PW13, LWW17}.

The paper is organized as follows.
In \Cref{sec:prelim}, we present the geometric setup and recall standard preliminary results, including stability properties.
\Cref{sec:asymp} establishes the uniform convergence of solutions to the 1D profile at infinity.
Finally, in \Cref{sec:comparison}, we provide the proof of \Cref{thm:main-order}.
To achieve this, we first derive a maximum principle for the linearized operator on unbounded domains and establish the existence of minimal barrier solutions.
We conclude the paper by presenting direct consequences of the main theorem, including strict ordering and comparisons with the saddle-shaped solution.

\section{Preliminaries}\label{sec:prelim}

We first recall the notion of saddle-shaped solutions associated with the Simons cone.

\begin{definition}[Saddle-shaped solution]\label{def:saddle}
We say that a bounded solution $u: \mathbb{R}^{2m} \to \mathbb{R}$ of \eqref{eq:main} is a \textit{saddle-shaped solution} if, with the radial coordinates $s = |x|$ and $t = |y|$ for $(x, y) \in \mathbb{R}^m \times \mathbb{R}^m$, it satisfies:
\begin{enumerate}[label=(\roman*)]
    \item $u$ depends only on the variables $s$ and $t$, written as $u(x, y) = u(s, t)$.
    \item $u > 0$ in the region $\mathcal{C}^+ = \{ (x, y) \in \mathbb{R}^{2m} : s > t \}$.
    \item $u$ is odd with respect to the diagonal, i.e., $u(s, t) = -u(t, s)$ in $\mathbb{R}^{2m}$.
\end{enumerate}
\end{definition}

Next, we introduce the geometric setup for regular minimal cones to define the asymptotic behavior of nodal sets. Let $\mathcal{C}$ be a regular minimal cone in $\mathbb{R}^n$ (that is, $\mathcal{C}$ has an isolated singularity at the origin). We denote the normal vector field defined over $\mathcal{C}\setminus B_R$ for some $R>0$ as $\eta$, and let $S$ be the cross section $\mathcal{C}\cap\mathbb{S}^{n-1}$ of the cone $\mathcal{C}$. We use the coordinate system $(z, r) = rz \in \mathcal{C} \subset \mathbb{R}^{n}$ on the cone, where $z \in S$ and $r>0$.

\begin{definition}[Asymptotic to a cone in the $C^1$ sense]\label{def:asymptotic_C1}
Let $\mathcal{C} \subset \mathbb{R}^n$ be a regular minimal cone. We say that the boundary $\Gamma = \partial \Omega$ of an open set $\Omega\subset \mathbb{R}^n$ is \textit{asymptotic to the cone $\mathcal{C}$ in the $C^1$ sense} if it satisfies the following condition:

There exist a radius $R > 0$ and a compact set $K \subset \mathbb{R}^n$ such that $\Gamma \setminus K$ can be represented as a normal graph over $\mathcal{C} \setminus \overline{B}_R$:
\begin{equation}\label{eq:normal_graph}
\Gamma \setminus K = \left\{ rz + h(r, z)\nu_{\mathcal{C}}(rz) \mid r > R, \, z \in S \right\},
\end{equation}
where $S = \mathcal{C} \cap \mathbb{S}^{n-1}$, $\nu_{\mathcal{C}}$ is a unit normal vector field on $\mathcal{C}$, and the function $h \in C^1((R, \infty) \times S)$ satisfies the decay condition
\begin{equation}\label{eq:C1_decay}
\lim_{r \to \infty} \sup_{z \in S} \left( |h(r, z)| + |\partial_r h(r, z)| + \frac{1}{r} |\nabla_S h(r, z)| \right) = 0.
\end{equation}
Here, $\nabla_S$ denotes the covariant derivative on the link $S$.
\end{definition}

We recall the strong maximum principle and the unique continuation property for the Allen--Cahn equation. The proofs are standard and are therefore omitted.

\begin{lemma}[Strong maximum principle]\label{lem:strong_mp}
Let $\Omega \subset \mathbb{R}^n$ be a domain. Suppose that $u$ and $v$ are solutions to \eqref{eq:main} in $\Omega$ with $u \le v$ in $\Omega$. If there exists a point $x_0 \in \Omega$ such that $u(x_0) = v(x_0)$, then $u \equiv v$ in $\Omega$.
\end{lemma}

\begin{lemma}[Unique continuation principle]\label{lem:ucp}
Let $\Omega \subset \mathbb{R}^n$ be a domain. Suppose that $u$ and $v$ are solutions to \eqref{eq:main} in $\Omega$. If $u = v$ on a non-empty open subset of $\Omega$, then $u \equiv v$ in $\Omega$.
\end{lemma}

The following proposition establishes the stability of positive solutions. This result corresponds to Remark 3.3 in \cite{CT10}, but we include a proof for completeness.

\begin{proposition}[Stability of positive solutions]\label{prop:pos_stability}
Let $\Omega \subset \mathbb{R}^n$ be a domain (possibly unbounded). Suppose that $u$ is a solution of the Allen--Cahn equation $-\Delta u = u - u^3$ in $\Omega$ such that $0 < u < 1$ in $\Omega$. Then $u$ is a stable solution in $\Omega$, i.e.,
\[
\int_{\Omega} \left( |\nabla \xi|^2 - (1 - 3u^2)\xi^2 \right) \, dx \ge 0 \quad \text{for all } \xi \in C_c^\infty(\Omega).
\]
\end{proposition}

\begin{proof}
Since $0 < u < 1$, we have $(1 - 3u^2)u < u - u^3 = -\Delta u$ in $\Omega$. 
Multiplying this inequality by $\xi^2/u$ for any $\xi \in C_c^\infty(\Omega)$ and integrating by parts, we find
\[
\int_{\Omega} (1 - 3u^2)\xi^2 \, dx \le \int_{\Omega} \frac{-\Delta u}{u} \xi^2 \, dx = \int_{\Omega} \nabla u \cdot \nabla \left( \frac{\xi^2}{u} \right) \, dx
\le \int_{\Omega} |\nabla \xi|^2 \, dx,
\]
where the identity $\nabla u \cdot \nabla (\frac{\xi^2}{u}) = |\nabla \xi|^2 - |u \nabla (\frac{\xi}{u})|^2$ is used in the last inequality.
\end{proof}

The following result is needed to rule out the trivial solution when we obtain global solutions via a compactness argument.

\begin{proposition}\label{prop:zero_instability}
Let $\Omega \subset \mathbb{R}^n$ be a domain. Suppose that $\Omega$ contains balls of arbitrarily large radius. Then the trivial solution $u \equiv 0$ is unstable in $\Omega$.
\end{proposition}

\begin{proof}
The linearized operator at $u \equiv 0$ corresponds to the quadratic form
\[
Q_0(\xi) := \int_{\Omega} \left( |\nabla \xi|^2 - \xi^2 \right) \, dx.
\]
To prove instability, we must find a test function $\xi$ such that $Q_0(\xi) < 0$.
For any $R > 0$, let $B_R \subset \Omega$ be a ball of radius $R$. Let $\phi_R$ be the first Dirichlet eigenfunction of $-\Delta$ on $B_R$ associated with the eigenvalue $\lambda_1(B_R)$. Extending $\phi_R$ by zero outside $B_R$, we have $\phi_R \in H^1_0(\Omega)$.
A direct computation shows
\[
Q_0(\phi_R) = \int_{B_R} \left( |\nabla \phi_R|^2 - \phi_R^2 \right) \, dx = (\lambda_1(B_R) - 1) \int_{B_R} \phi_R^2 \, dx.
\]
Since $\lambda_1(B_R) \propto R^{-2}$, we can choose $R$ sufficiently large such that $\lambda_1(B_R) < 1$. This implies $Q_0(\phi_R) < 0$, completing the proof.
\end{proof}

\section{Asymptotic behavior of solutions}\label{sec:asymp}

In this section, we establish the asymptotic behavior of solutions defined on a domain $\Omega \subset \mathbb{R}^{n}$ whose boundary $\Gamma = \partial \Omega$ is asymptotic to a minimal cone $\mathcal{C}$. In particular, we consider a solution $u$ that is positive in $\Omega$ and vanishes on $\Gamma$.

Recall that the Simons cone $\mathcal{C}_{m,m}$ is defined by$$\mathcal{C}_{m,m} = \{ (x, y) \in \mathbb{R}^m \times \mathbb{R}^m : |x| = |y| \}.$$Consider an entire solution of \eqref{eq:main} whose nodal set is asymptotic to $\mathcal{C}_{m,m}$. By restricting this solution to one of its nodal domains, where the solution is positive, we can apply the results established in this section. This generalizes the asymptotic analysis in \cite{CT10}[Theorem 1.6], where the authors established the result for entire saddle-shaped solutions whose nodal set is \textit{exactly} the Simons cone.

\begin{lemma}\label{lem:tangent}
Let $\Omega \subset \mathbb{R}^{n}$ be a domain with boundary $\Gamma = \partial \Omega$ asymptotic to a minimal hypercone $\mathcal{C}$ in the $C^1$ sense (see Definition~\ref{def:asymptotic_C1}).
Let $\{x_k\} \subset \Gamma$ be a sequence with $|x_k| \to \infty$ such that the inward unit normal $\nu(x_k)$ converges to $\nu_\infty \in \mathbb{S}^{n-1}$.
Set $H := \{ y \cdot \nu_\infty = 0 \}$ and $H^+ := \{ y \cdot \nu_\infty > 0 \}$.
Then, for any $R > 0$, we have 
\begin{enumerate}
    \item $d_{\mathcal{H}} \big( (\Gamma - x_k) \cap B_{R}, H \cap B_{R} \big) \to 0$,
    \item $d_{\mathcal{H}} \big( (\Omega - x_k) \cap B_{R}, H^+ \cap B_{R} \big) \to 0$,
\end{enumerate}
as $k \to \infty$, where $d_{\mathcal{H}}$ denotes the Hausdorff distance.
\end{lemma}
\begin{proof}
Since (2) follows immediately from (1), we prove (1).

Let $\Gamma_k := \Gamma - x_k$ and $\mathcal{C}_k := \mathcal{C} - p_k$, where $x_k = p_k + h(p_k)\nu_{\mathcal{C}}(p_k)$ as in Definition~\ref{def:asymptotic_C1}. By the triangle inequality for the Hausdorff distance $d_{\mathcal{H}}$, it suffices to show that $\Gamma_k$ converges to $\mathcal{C}_k$, and $\mathcal{C}_k$ converges to $H$ in $B_R(0)$.

First, we compare $\Gamma$ with $\mathcal{C}$. Since $|x_k| \to \infty$, the ball $B_R(x_k)$ is contained in the region where $r \gg 1$. The decay condition \eqref{eq:C1_decay} implies that $|h(r, z)| \to 0$ uniformly as $r \to \infty$. Consequently, $\Gamma$ becomes arbitrarily close to $\mathcal{C}$ within the fixed ball $B_R$ relative to the shifted frames. Thus,
\begin{align}\label{conv:GkCk}
d_{\mathcal{H}} \big( \Gamma_k \cap B_R, \mathcal{C}_k \cap B_R \big) \to 0 \quad \text{as } k \to \infty.    
\end{align}

Next, we compare $\mathcal{C}_k$ with $H$. Since the principal curvatures of $\mathcal{C}$ decay as $O(1/|p_k|)$, the cone becomes asymptotically flat. Therefore, 
\begin{align}\label{conv:CkHk}
d_{\mathcal{H}} \big( \mathcal{C}_k \cap B_R, H_k \cap B_R \big) \to 0 \quad \text{as } k \to \infty,
\end{align}
where $H_k := \{ y \in \mathbb{R}^n : y \cdot \nu_{\mathcal{C}}(p_k) = 0 \}$ is the tangent hyperplane of $\mathcal{C}$ at $p_k$.
Hence it suffices to show that $\nu_{\mathcal{C}}(p_k) \rightarrow \nu_\infty$. 

We scale $\Gamma$ to $\Gamma/r_k$, where $r_k=|p_k|$. Note that the cone $\mathcal{C}$ is invariant with respect to the scaling. $\Gamma/r_k$ is still a normal graph with respect to $\mathcal{C}$ represented by $h_k'(r,z) = h(r_kr,z )/r_k$. Near $(1,p_k/r_k)$, the scaled hypersurface $\Gamma/r_k$ is a graph over a fixed compact region of $\mathcal{C}$. Since $|h_k'|, |Dh_k'| \rightarrow 0$ as $k \rightarrow +\infty$ by \eqref{eq:C1_decay}, $\nu_\mathcal{C}(p_k)\rightarrow \nu_{\infty}$ as $\Gamma/r_k$ is a $C^1$-perturbation of $\mathcal{C}$. Hence
\[
d_{\mathcal{H}} \big( H_k \cap B_R, H \cap B_R \big) \to 0 \quad \text{as } k \to \infty
\]
which, together with \eqref{conv:GkCk} and \eqref{conv:CkHk}, completes the proof.
\end{proof}

We denote by $u_0(z) := \tanh(z/\sqrt{2})$ the 1D heteroclinic solution. We also define the asymptotic profile function $U$ associated with the cone $\mathcal{C}$ by
\[
U(x) := u_0\left( \mathrm{dist}(x, \mathcal{C}) \right).
\]

\begin{theorem}\label{thm:asym}
Let $\Omega \subset \mathbb{R}^{n}$ be a domain whose boundary $\Gamma := \partial \Omega$ is asymptotic to a minimal cone $\mathcal{C}$ at infinity in the $C^1$ sense (see Definition~\ref{def:asymptotic_C1}). Let $u$ be a bounded solution of \eqref{eq:main} in $\Omega$ such that $u = 0$ on $\Gamma = \partial \Omega$ and $u > 0$ in $\Omega$.
Then,
\[
\|u - U\|_{L^{\infty}(\Omega \setminus B_R)} + \|\nabla u - \nabla U\|_{L^{\infty}(\Omega \setminus B_R)} \rightarrow 0 \quad \text{as } R \rightarrow \infty.
\]
\end{theorem}

\begin{proof}
We proceed by contradiction. Suppose that the assertion of the theorem fails. Then, there exist a constant $\varepsilon > 0$ and a sequence of points $\{x_k\} \subset \Omega$ such that $|x_k| \to \infty$ and
\begin{equation}\label{contra}
|u(x_k) - U(x_k)| + |\nabla u(x_k) - \nabla U(x_k)| \ge \varepsilon.
\end{equation}
Since $u$ is bounded and thus $|u|\le 1$ by the maximum principle, standard elliptic estimates imply that $\nabla u$ is also uniformly bounded in $\Omega$ (c.f. \cite{Modica85}).

We distinguish two cases.

\vspace{1em}
\noindent \textbf{Case 1.} The sequence $d_k := \mathrm{dist}(x_k, \Gamma)$ is unbounded.

Up to a subsequence, assume $d_k \to \infty$. Since $\Gamma$ is asymptotic to $\mathcal{C}$, the distance from $x_k$ to the cone $\mathcal{C}$ also tends to infinity. In the domain $\Omega$, this implies that $U(x_k) \to 1$ and $\nabla U(x_k) \to 0$.
Consequently, \eqref{contra} yields
\begin{equation}\label{case1}
|u(x_k) - 1| + |\nabla u(x_k)| \ge \frac{\varepsilon}{2}
\end{equation}
for sufficiently large $k$.

Let $v_k(y) := u(y + x_k)$. Since $d_k = \mathrm{dist}(x_k, \Gamma) \to \infty$, for any fixed $R > 0$, the ball $B_R(x_k)$ is strictly contained in $\Omega$ for all sufficiently large $k$.
Thus, the sequence $v_k$ is well-defined on balls expanding to the whole space $\mathbb{R}^{n}$. By standard elliptic estimates, up to a subsequence, $v_k$ converges in $C^2_{\text{loc}}(\mathbb{R}^{n})$ to a limit function $v$. Hence, $v$ is an entire solution satisfying
\begin{equation}\label{solpos}
\begin{cases}
-\Delta v = v - v^3 & \text{in } \mathbb{R}^{n}, \\
0 \le v \le 1 & \text{in } \mathbb{R}^{n}.
\end{cases}
\end{equation}

Applying Liouville-type results (e.g., \cite[Proposition 1.14]{BHN05}), $v$ must be identically constant, i.e., $v \equiv 0$ or $v \equiv 1$.
If $v \equiv 1$, then $u(x_k) \to 1$ and $\nabla u(x_k) \to 0$, which contradicts \eqref{case1}. 
This leaves only the case $v \equiv 0$. However, since stability is preserved under local convergence, it follows from \Cref{prop:pos_stability} that the limit $v$ must be a stable solution in $\mathbb{R}^{n}$.
This contradicts Proposition~\ref{prop:zero_instability}, which asserts that the trivial solution is unstable in any domain containing arbitrarily large balls (such as $\mathbb{R}^{n}$).
Thus, Case 1 cannot occur.

\vspace{1em}
\noindent \textbf{Case 2.} The sequence $d_k := \mathrm{dist}(x_k, \Gamma)$ is bounded.

Up to a subsequence, assume $d_k \to d \ge 0$. Let $x_k^0 \in \Gamma$ be a point realizing the distance, i.e., $|x_k - x_k^0| = d_k$. Since $x_k \in \Omega$, we can write
\[
x_k = x_k^0 + d_k \nu(x_k^0),
\]
where $\nu(x_k^0)$ denotes the unit inward normal to $\Gamma$ at $x_k^0$.
Since $\{\nu(x_k^0)\}$ lies on the unit sphere, up to a subsequence, there exists a unit vector $\nu \in \mathbb{S}^{n-1}$ such that $\nu(x_k^0) \to \nu$ as $k \to \infty$.

Let $w_k(y) := u(y + x_k^0)$ defined on the shifted domains $\Omega_k := \Omega - x_k^0$.
Since $\Gamma$ is asymptotic to the minimal cone at infinity, the sequence of domains $\Omega_k$ converges locally to a half-space $H^+$.
Standard elliptic estimates imply that $w_k$ converges locally in $C^2(H^+)$ to a limit $w$ defined in $H^+$.

To determine the boundary condition on $H = \partial H^+$, consider any point $y \in H^+$.
Using the mean value theorem and the uniform boundedness of $|\nabla u|$, we have the estimate
\[
|w_k(y)| = |u(y + x_k^0)| \le \|\nabla u\|_{L^\infty} \mathrm{dist}(y + x_k^0, \Gamma) = \|\nabla u\|_{L^\infty} \mathrm{dist}(y, \partial \Omega_k).
\]
Passing to the limit $k \to \infty$, since $\partial \Omega_k \to H$ locally uniformly, we obtain
\[
|w(y)| \le C \, \mathrm{dist}(y, H) \quad \text{for all } y \in H^+,
\]
where $C$ is a universal constant.
This inequality implies that $w(y) \to 0$ as $\mathrm{dist}(y, H) \to 0$. Thus, $w$ extends continuously to the boundary $H$ with
\[
w(y) = 0 \quad \text{for } y \in H.
\]

Combined with the positivity of $u$ in $\Omega$, the limit $w$ solves the Dirichlet problem in the half-space:
\begin{equation}\label{claimH}
\begin{cases}
-\Delta w = w - w^3 & \text{in } H^+, \\
w \ge 0 & \text{in } H^+, \\
w = 0 & \text{on } H.
\end{cases}
\end{equation}
By \Cref{prop:pos_stability}, $u$ is stable in $\Omega$, and thus the limit $w$ is a stable solution in $H^+$.
According to \Cref{prop:zero_instability}, the trivial solution is unstable, which implies $w \not\equiv 0$.
By the strong maximum principle, $w$ is strictly positive. Since $w$ is a bounded solution in a half-space, we apply the rigidity result of Angenent \cite{Angenent85} to deduce that $w$ is the unique 1D profile depending only on the distance to the boundary.
Thus,
\[
w(y) = u_0(y \cdot \nu) \quad \text{for all } y \in H^+,
\]
where $\nu$ is the inner unit normal to $H$.

We now derive a contradiction to \eqref{contra}. Recall that $x_k = x_k^0 + d_k \nu_k^0$ with $\nu_k^0 \to \nu$. Using the local uniform convergence $w_k \to w$, we estimate $u(x_k)$:
\[
u(x_k) = w_k(d_k \nu_k^0) = w(d_k \nu) + o(1) = u_0(d_k) + o(1).
\]
Since $u_0(d_k) = U(x_k) + o(1)$, we have $u(x_k) = U(x_k) + o(1)$.
Similarly, $\nabla u(x_k) = \nabla U(x_k) + o(1)$.
This implies $|u(x_k) - U(x_k)| + |\nabla u(x_k) - \nabla U(x_k)| \to 0$, contradicting \eqref{contra}.
\end{proof}

\section{Proof of Theorem~\ref{thm:main-order}}\label{sec:comparison}
In this section, we present the proof of \Cref{thm:main-order}. To do so, we first establish a maximum principle on unbounded domains and prove the existence of minimal solutions. These results generalize Proposition 1.3 and Lemma 3.1 in \cite{Cabre12} for saddle-shaped solutions to our setting.

\begin{lemma}\label{lem:max-principle-unbounded}
Let $\mathcal{O} \subset \mathbb{R}^n$ be a domain whose boundary $\Gamma = \partial\mathcal{O}$ is asymptotic to a minimal cone $\mathcal{C}$ in the $C^1$ sense (see Definition~\ref{def:asymptotic_C1}).
Let $u$ be an entire solution of \eqref{eq:main} with positive phase $\mathcal{O}$, i.e., $\mathcal{O} = \{ x \in \mathbb{R}^n : u(x) > 0 \}$.
Suppose that $v \in C^2(\mathcal{O}) \cap C^0(\overline{\mathcal{O}})$ is a subsolution of the linearized operator at $u$, satisfying
\[
\Delta v + (1 - 3u^2)v \ge 0 \quad \text{in } \mathcal{O}, \qquad v \le 0 \quad \text{on } \Gamma,
\]
and
\[
\limsup_{\substack{x \in \mathcal{O} \\ |x| \to \infty}} v(x) \le 0.
\]
Then $v \le 0$ in $\mathcal{O}$.
\end{lemma}

\begin{proof}
Suppose for contradiction that $v(x_0) > 0$ at some $x_0 \in \mathcal{O}$.
Let $t(x) := \mathrm{dist}(x, \Gamma)$ be the distance function to the boundary.
Fix $\varepsilon \in (0, t(x_0)]$ (to be determined later).
We partition the domain into a boundary strip $N_\varepsilon$ and an interior set $\mathcal{O}_\varepsilon$:
\[
N_\varepsilon = \{ x \in \mathcal{O} : 0 < t(x) < \varepsilon \}, \qquad
\mathcal{O}_\varepsilon = \{ x \in \mathcal{O} : t(x) > \varepsilon \}.
\]

By the one-dimensional heteroclinic asymptotics (\Cref{thm:asym}), there exists $\kappa_\varepsilon > 0$ such that
\[
u(x) \ge \kappa_\varepsilon \quad \text{for all } x \in \mathcal{O}_\varepsilon.
\]
Define $w := v/u$ on $\mathcal{O}_\varepsilon$ and set $\alpha := \sup_{\overline{\mathcal{O}_\varepsilon}} w$.
Since $u \ge \kappa_\varepsilon > 0$ on $\mathcal{O}_\varepsilon$, the hypotheses $v \le 0$ on $\Gamma$ and $\limsup_{|x| \to \infty} v(x) \le 0$ imply $\limsup_{|x| \to \infty} w(x) \le 0$. Since $v(x_0) > 0$, we have $\alpha > 0$. Moreover, the supremum is attained at some interior point $x_1 \in \mathcal{O}_\varepsilon$, i.e., $w(x_1) = \alpha$.

Consider the function $\phi := v - \alpha u$. We claim that $\phi \le 0$ in $N_\varepsilon$.
First, observe that $L_u \phi = L_u v - \alpha L_u u$. Since $L_u v \ge 0$ by assumption and
\[
L_u u = \Delta u + (1-3u^2)u = -(u - u^3) + (1-3u^2)u = -2u^3 < 0,
\]
we have $L_u \phi \ge 2\alpha u^3 > 0$ in $N_\varepsilon$.
The boundary conditions are:
\begin{itemize}
    \item On $\Gamma$, $v \le 0$ and $u = 0$, so $\phi \le 0$.
    \item On $\partial N_\varepsilon \cap \mathcal{O} = \partial \mathcal{O}_\varepsilon$, $v/u \le \alpha$ implies $\phi \le 0$.
\end{itemize}
 
Although the region $N_\varepsilon$ may not be a `narrow domain' due to the irregularity of the boundary $\Gamma$, the maximum principle for $L_u$ remains valid provided that the complement of the domain satisfies a uniform density estimate (see Berestycki, Nirenberg, and Varadhan \cite[Theorem 2.5]{BNV94}).

Suppose for contradiction that $\sup_{N_\varepsilon} \phi > 0$.
Given the boundary condition $\phi \le 0$ on $\partial N_\varepsilon \cap \mathcal{O}$ and the decay condition at infinity, there exists a maximum point $x_0 \in N_\varepsilon$ such that
\[
    \phi(x_0) = \sup_{N_\varepsilon} \phi > 0.
\]
We construct a comparison function $\psi(x) := a(\varepsilon^2 - |x-x_0|^2)$ with a constant $a > 0$ to be determined.
Consider the set $D := N_\varepsilon \cap \{ \phi > 0 \}$. In this region, since $\Delta \phi \ge -(1-3u^2)\phi$ and $1-3u^2 \le 1$, we have
\[
    \Delta(\phi - \psi) = \Delta \phi + 2na \ge -(1-3u^2)\phi + 2na \ge -\phi(x_0) + 2na.
\]
Choosing $a := \phi(x_0)/2n$, we have $\Delta(\phi - \psi) \ge 0$ in $D \cap B_{2\varepsilon}(x_0)$.
Since $\phi - \psi \le \phi \le 0$ on $\partial D \cap B_{2\varepsilon}(x_0)$, we extend the function $(\phi - \psi)_+$ to be zero on $B_{2\varepsilon}(x_0) \setminus D$. This extension is continuous and subharmonic in the viscosity sense in $B_{2\varepsilon}(x_0)$.

Applying the mean value property for subharmonic functions on $B_{2\varepsilon}(x_0)$, we obtain
\[
    \phi(x_0) - a\varepsilon^2 = (\phi - \psi)(x_0) \le \frac{1}{|B_{2\varepsilon}(x_0)|} \int_{B_{2\varepsilon}(x_0)} (\phi - \psi)_+ \, dx.
\]
Since $(\phi - \psi)_+ \le \phi \le \phi(x_0)$ and the support is contained in $D \cap B_{2\varepsilon}(x_0)$,
\[
    \phi(x_0) \left( 1 - \frac{\varepsilon^2}{2n} \right) \le \frac{|D \cap B_{2\varepsilon}(x_0)|}{|B_{2\varepsilon}(x_0)|} \phi(x_0).
\]
Dividing by $\phi(x_0) > 0$ and rearranging yield a bound for the density of the complement set:
\[
    \frac{|B_{2\varepsilon}(x_0) \setminus D|}{|B_{2\varepsilon}(x_0)|} \le \frac{\varepsilon^2}{2n}.
\]
Since $x_0 \in N_\varepsilon$, there exists a boundary point $\bar{x}_0 \in \Gamma$ such that $|x_0 - \bar{x}_0| < \varepsilon$. Observing that $B_\varepsilon(\bar{x}_0) \setminus N_\varepsilon \subset B_{2\varepsilon}(x_0) \setminus D$ (since $\phi \le 0$ on $\Gamma$ and outside $\mathcal{O}$), we deduce
\[
    \frac{|B_\varepsilon(\bar{x}_0) \setminus N_\varepsilon|}{|B_{2\varepsilon}(x_0)|} \le \frac{\varepsilon^2}{2n}.
\]

To derive a contradiction, we rely on the geometric properties of the nodal set $\Gamma = \{u=0\}$.
Recall that the solution $u$ to the Allen--Cahn equation is real analytic. The nodal set $\Gamma$ decomposes into a regular part $\mathcal{R} = \{ x \in \Gamma : |\nabla u(x)| \neq 0 \}$ and a singular part $\mathcal{S} = \{ x \in \Gamma : \nabla u(x) = 0 \}$.

We establish a uniform density estimate for the complement of the domain, $\mathcal{O}^c$.
By \Cref{thm:asym}, there exists a large radius $R_0 > 0$ such that for $|x| > R_0$ such that $\Gamma\setminus \overline B_R \subset \mathcal{R}$. Thus the uniform density estimate holds: there exist $\varepsilon_1, c_1 > 0$ such that
\[
\frac{|B_\rho(x) \setminus \mathcal{O}|}{|B_\rho(x)|} \ge c_1 \quad \text{for all } x \in \Gamma \cap \{|x| > R_0\} \text{ and } \rho < \varepsilon_1.
\]
Next, consider the compact part $K := \Gamma \cap \overline{B_{R_0}}$.
For any $x \in K$, let $d(x)$ denote the vanishing order of $u$ at $x$.
If $x \in \mathcal{R}$, then $d(x)=1$, and $\Gamma$ is locally a $C^1$ hypersurface, implying $\lim_{\rho \to 0} |B_\rho(x) \setminus \mathcal{O}|/|B_\rho(x)| = 1/2$.
If $x \in \mathcal{S}$, we claim that the vanishing order satisfies $\sup_{x \in \mathcal{S} \cap K} d(x) \le C < \infty$. Indeed, if not, there would exist a sequence $x_k \in \mathcal{S} \cap K$ such that $d(x_k) \to \infty$. Up to a subsequence, $x_k \to x_* \in K$. By the smoothness of $u$, this would imply $u$ vanishes to infinite order at $x_*$. However, by the unique continuation principle for analytic functions, $u$ must be identically zero, which contradicts $u > 0$ in $\mathcal{O}$.
Since the vanishing order is bounded, it follows from \cite{Han94} that the local geometry of $\Gamma$ at any singular point is governed by homogeneous harmonic polynomials of finite degree, i.e., $u(x)=P(x)+o(x^d)$, where $P$ is a nonzero homogeneous polynomial with degree $d$. This gives that the complement $\mathcal{O}^c$ has positive density at each point. Specifically,
\[
\liminf_{\rho \to 0} \frac{|B_\rho(x) \setminus \mathcal{O}|}{|B_\rho(x)|} > 0 \quad \text{for all } x \in K.
\]
Hence, there exist uniform constants $\varepsilon_2, c_2 > 0$ such that 
\[
\frac{|B_\rho(x) \setminus \mathcal{O}|}{|B_\rho(x)|} \ge c_2 \quad \text{for all } x \in K \text{ and } \rho < \varepsilon_2.
\]
Indeed, if there exist a sequence of points $x_k \in K$ and radii $\rho_k \searrow 0$ such that
\[
\frac{|B_{\rho_k}(x_k) \setminus \mathcal{O}|}{|B_{\rho_k}(x_k)|} < \frac{1}{k},
\]
then since $x_k$ converges to some $x_0 \in K$ up to a subsequence, the lower density of $\mathcal{O}^c$ at $x_0$ is zero, which contradicts the strictly positive pointwise density assumption.

Combining the estimates for the compact part and the asymptotic tail, we conclude that there exists a universal constant $c > 0$ such that
\[
\frac{|B_{\varepsilon}(\bar{x}_0) \setminus \mathcal{O}|}{|B_{\varepsilon}(\bar{x}_0)|} \ge c \quad \text{for all } \bar{x}_0 \in \Gamma.
\]
Recall the contradiction inequality derived earlier:
\[
\frac{|B_\varepsilon(\bar{x}_0) \setminus N_\varepsilon|}{|B_{2\varepsilon}(x_0)|} \le \frac{\varepsilon^2}{2n}.
\]
Since $B_\varepsilon(\bar{x}_0) \setminus \mathcal{O} \subset B_\varepsilon(\bar{x}_0) \setminus N_\varepsilon$ and $|B_{2\varepsilon}(x_0)| = 2^n |B_\varepsilon(\bar{x}_0)|$, we have
\[
\frac{c}{2^n} \le \frac{|B_\varepsilon(\bar{x}_0) \setminus \mathcal{O}|}{|B_{2\varepsilon}(x_0)|} \le \frac{|B_\varepsilon(\bar{x}_0) \setminus N_\varepsilon|}{|B_{2\varepsilon}(x_0)|} \le \frac{\varepsilon^2}{2n}.
\]
This inequality $c/2^n \le \varepsilon^2/2n$ leads to a contradiction if we choose $\varepsilon$ sufficiently small (specifically, $\varepsilon < \sqrt{2n c / 2^n}$). Thus, the assumption $\sup_{N_\varepsilon} \phi > 0$ is false, completing the proof of $\phi\le 0$ on $N_\varepsilon$.

Combining this with the definition of $\alpha$ on $\mathcal O_{\varepsilon}$, we conclude that
$$
w=\frac{v}{u}\le \alpha \ \text{ in } \mathcal O, \qquad \text{and} \qquad w(x_1)=\alpha.
$$
In particular, $w$ attains an interior maximum in $\mathcal O$.

Next we compute the equation satisfied by $w$ in $\mathcal O$. Writing $v=uw$ and using $\Delta(uw)=u\,\Delta w+2\nabla u\cdot\nabla w+w\,\Delta u$, we obtain
$$
\begin{aligned}
0\le L_u(uw)
&= \Delta(uw)+(1-3u^2)uw \\
&= u\Big(\Delta w+2\,\frac{\nabla u}{u}\cdot\nabla w + w\Big(\frac{\Delta u}{u}+1-3u^2\Big)\Big).
\end{aligned}
$$
Since $u$ solves $\Delta u+u-u^3=0$, the bracket becomes
$$
\Delta w + 2\,\nabla(\log u)\cdot\nabla w - 2u^2\,w \;\ge\; 0\quad\text{in }\mathcal O .
$$
At the interior maximum point $x_1$ we have $\nabla w(x_1)=0$ and $\Delta w(x_1)\le 0$, hence
$$
0 \le \Delta w(x_1) - 2u(x_1)^2\,w(x_1) \le -\,2u(x_1)^2\,\alpha < 0,
$$
a contradiction. Therefore $v\le 0$ in $\mathcal O$.
\end{proof}

\begin{lemma}\label{lem:minsol}
Let $\mathcal{O} \subset \mathbb{R}^{n}$ be an open set whose boundary $\Gamma = \partial \mathcal{O}$ is asymptotic to a minimal cone $\mathcal{C}$ in the $C^1$ sense.
Suppose that $u_1, u_2 \in C^2(\mathcal{O}) \cap C^0(\overline{\mathcal{O}})$ are solutions of \eqref{eq:main} in $\mathcal{O}$ which are positive in $\mathcal{O}$. If one of $u_1, u_2$ vanishes on $\partial\mathcal{O} = \Gamma$,
then there exists a positive solution $u_0 \in C^2(\mathcal{O}) \cap C^0(\overline{\mathcal{O}})$ of \eqref{eq:main} such that
\[
0 < u_0 \le \min\{u_1, u_2\} \quad \text{in } \mathcal{O}, \qquad \text{and} \quad u_0 = 0 \quad \text{on } \Gamma.
\]
\end{lemma}

\begin{proof}
Let $\psi:=\min\{u_1,u_2\}$. By the hypotheses, $0<\psi<1$ in $\mathcal O$ and $\psi=0$ on $\Gamma$. Since $\psi$ is the minimum of two solutions, it is a supersolution of the Allen–Cahn equation in the viscosity sense.

Fix $R>1$ and set $D_R:=\mathcal O\cap B_R$. The domain $D_R$ is a bounded open set. The boundary data
$$
u=0\ \text{on } \Gamma\cap B_R,\qquad u=\psi\ \text{on } \partial B_R\cap\mathcal O
$$
are compatible along $\Gamma\cap\partial B_R$ because $\psi|_\Gamma\equiv0$.
Using the sub/supersolution pair $0\le\psi$, we obtain a solution
$$
u_R\in C^2(D_R)\cap C^0(\overline{D_R})\quad\text{of \eqref{eq:main} with}\quad 0\le u_R\le\psi\ \text{in }D_R.
$$

By interior Schauder estimates, the family $\{u_R\}_{R>1}$ is uniformly bounded in $C^{2,\alpha}(K)$ for every compact set $K \Subset \mathcal O$, with constants independent of $R$.
Hence there exists a subsequence $\{u_{R_j}\}$ converging to a function $u_0$ in $C^{2}_{\mathrm{loc}}(\mathcal O)$.
Since $0 \le u_0 \le \psi$ in $\mathcal O$, the limit $u_0$ extends continuously to $\Gamma$ with value $0$. Thus, $u_0 \in C^2(\mathcal{O})\cap C^0(\overline{\mathcal{O}})$ solves \eqref{eq:main} in $\mathcal{O}$ with $\{u_0 = 0\} \supset \Gamma$.

It remains to prove that $u_0 > 0$ in $\mathcal{O}$. As before, the strong maximum principle implies that either $u_0 > 0$ or $u_0 \equiv 0$ in $\mathcal{O}$.
Suppose for contradiction that $u_0 \equiv 0$.
Since $u_R$ is a positive solution in $D_R$ satisfying $0<u_R\le \psi<1$, $u_R$ is a stable solution in $D_R$ by \Cref{prop:pos_stability}. Passing to the limit $R \to \infty$ in the stability inequality, we have $u_0$ is also a stable solution in $\mathcal{O}$. However, since $\mathcal{O}$ contains balls of arbitrarily large radius, by \Cref{prop:zero_instability}, the trivial solution is unstable. This is a contradiction. Therefore, $u_0 > 0$ in $\mathcal{O}$, and consequently $\{u_0 = 0\} = \Gamma$.
\end{proof}

We are now in a position to prove the main theorem.

\begin{proof}[Proof of Theorem~\ref{thm:main-order}]
Let $u_1$ and $u_2$ be two solutions of \eqref{eq:main} such that their nodal sets are asymptotic to the minimal cone $\mathcal{C}$ in $C^1$ sense. Suppose that $\Omega_1 \subseteq \Omega_2$.
By Lemma~\ref{lem:minsol}, there exists a minimal solution $u \in C^2(\Omega_1) \cap C^0(\overline{\Omega}_1)$ satisfying
\begin{equation}\label{eq:u-min-prop}
0 < u(x) \le \min \{ u_1(x), u_2(x) \} \quad \text{for all } x \in \Omega_1,
\end{equation}
 and $u(x)=0$ for $x \in \partial\Omega_1$. Note that $u_2$ is positive in $\Omega_1$ (since $\Omega_1 \subseteq \Omega_2$), so $\min\{u_1, u_2\}$ is positive and \eqref{eq:u-min-prop} is well-defined.

We define the difference function $v := u_1 - u$ in $\overline{\Omega}_1$. From \eqref{eq:u-min-prop}, we have $v \ge 0$ in $\Omega_1$. Our goal is to show that $v \equiv 0$.

First, we verify that $v$ is a subsolution of the linearized operator at $u$. Let $L_u := \Delta + (1 - 3u^2)$. Since both $u_1$ and $u$ satisfy the Allen--Cahn equation \eqref{eq:main} in $\Omega_1$, we have
\[
\begin{aligned}
L_u v &= \Delta(u_1 - u) + (1-3u^2)(u_1 - u) \\
&= u_1^3 - u^3 - 3u^2(u_1 - u) = (u_1 - u)^2 (u_1 + 2u).
\end{aligned}
\]
Since $u_1 > 0$ and $u > 0$ in $\Omega_1$, it follows that $L_u v \ge 0$ in $\Omega_1$. Thus, $v$ is a subsolution for $L_u$ in $\Omega_1$.

Next, we consider the boundary and asymptotic behavior. Since $u_1 = 0$ and $u = 0$ on $\partial \Omega_1$, we clearly have $v = 0$ on $\partial \Omega_1$. 

Regarding the behavior at infinity, we apply Theorem~\ref{thm:asym}, which ensures that any solution whose nodal set is asymptotic to the minimal cone $\mathcal{C}$ converges uniformly to a canonical profile $U_\mathcal{C}$. Since both $u_1$ and $u$ belong to this class of solutions, we have
\[
\lim_{|x|\to\infty} |u_1(x) - U_\mathcal{C}(x)| = 0 \quad \text{and} \quad \lim_{|x|\to\infty} |u(x) - U_\mathcal{C}(x)| = 0.
\]
By the triangle inequality, it follows that
\[
\limsup_{|x|\to\infty, \, x\in\Omega_1} |v(x)| \le \limsup_{|x|\to\infty} |u_1(x) - U_\mathcal{C}(x)| + \limsup_{|x|\to\infty} |u(x) - U_\mathcal{C}(x)| = 0.
\]

Since $v$ vanishes on $\partial\Omega_1$ and at infinity, and satisfies $L_u v \ge 0$ in $\Omega_1$, we can apply the maximum principle for unbounded domains (Lemma~\ref{lem:max-principle-unbounded}) to conclude that $v \le 0$ in $\Omega_1$. Combining this with the fact that $v \ge 0$ by construction, we obtain $v \equiv 0$ in $\Omega_1$, which implies $u_1 \equiv u$ in $\overline{\Omega}_1$. Recalling \eqref{eq:u-min-prop}, we deduce
\begin{equation}\label{eq:pos_phase}
u_1 \le u_2 \quad \text{in } \Omega_1 = \{u_1 > 0\}.
\end{equation}

Next, we consider the negative phases. Note that the inclusion $\{u_1 > 0\} \subseteq \{u_2 > 0\}$ implies
\[
\{u_2 \le 0\} \subseteq \{u_1 \le 0\}.
\]
We claim that this implies the inclusion of the strictly negative phases, i.e., $\{u_2 < 0\} \subseteq \{u_1 < 0\}$.
Suppose, for the sake of contradiction, that this is not true. Then there exists a point $x_0$ such that $u_2(x_0) < 0$ but $u_1(x_0) \ge 0$. Since $x_0 \in \{u_2 \le 0\} \subseteq \{u_1 \le 0\}$, it must be that $u_1(x_0) = 0$.
Since the set $\{u_2 < 0\}$ is open, $x_0$ is an interior local maximum of $u_1$ (as $u_1 \le 0$ in this set).
Applying the strong maximum principle to the equation $-\Delta u_1 = u_1-u_1^3$, we conclude that $u_1 \equiv 0$ in the connected component containing $x_0$, and by unique continuation, $u_1 \equiv 0$ globally. This contradicts the assumption that $u_1$ is a nonconstant solution.
Thus, we have $\{u_2 < 0\} \subseteq \{u_1 < 0\}$.

Now, let $v_1 := -u_2$ and $v_2 := -u_1$. Since the nonlinearity $f(u) = u - u^3$ is odd, $v_1$ and $v_2$ are also solutions to \eqref{eq:main}. The inclusion proved above translates to $\{v_1 > 0\} \subseteq \{v_2 > 0\}$.
We can thus apply the exact same argument as in the positive phase case to the pair $v_1$ and $v_2$. This yields $v_1 \le v_2$ in $\{v_1 > 0\}$, which corresponds to $-u_2 \le -u_1$ in $\{u_2 < 0\}$. Equivalently,
\begin{equation}\label{eq:neg_phase}
u_1 \le u_2 \quad \text{in }  \{u_2 < 0\}.
\end{equation}

Finally, in the remaining region where $u_1 \le 0$ and $u_2 \ge 0$, the inequality $u_1 \le u_2$ holds trivially. Combining this with \eqref{eq:pos_phase} and \eqref{eq:neg_phase}, we conclude that $u_1 \le u_2$ in the whole space $\mathbb{R}^n$.

The dichotomy that either $u_1 \equiv u_2$ or $u_1 < u_2$ everywhere then follows directly from the strong maximum principle. This completes the proof.
\qedhere
\end{proof}

We close this section with two direct corollaries of \Cref{thm:main-order}.

\begin{corollary}[Strict ordering and disjoint nodal sets]\label{cor:strict}
Let $u_1$ and $u_2$ be solutions satisfying the assumptions of Theorem~\ref{thm:main-order}, and let $\Omega_i = \{u_i > 0\}$.
If $\Omega_1 \subsetneq \Omega_2$, then $u_1 < u_2$ in $\mathbb{R}^n$.
Consequently, the nodal sets are disjoint, i.e., $\partial \Omega_1 \cap \partial \Omega_2 = \emptyset$.
\end{corollary}

A significant consequence of \Cref{cor:strict} is the rigidity of the nodal set against \textit{one-sided} local perturbations.
If a solution's nodal set is asymptotic to a minimal cone, it is impossible to deform the nodal set solely within a compact set in one direction (i.e., strictly enlarging or shrinking the positive phase) to form a nodal set of another solution.

\medskip

Finally, we apply \Cref{thm:main-order} to the saddle-shaped solution $u_s$. Since the nodal set of $u_s$ is exactly the Simons cone, it serves as a global barrier for any solution whose positive phase is contained within one side of the cone.

\begin{corollary}[Comparison with the saddle-shaped solution]\label{cor:comparison-saddle}
Let $\mathcal{C}_{m,m}$ be the Simons cone defined by \eqref{eq:simonscone} and let $u_s$ be the saddle-shaped solution of \eqref{eq:main} (see Definition~\ref{def:saddle}).
Suppose that $u$ is a nonconstant solution of \eqref{eq:main} whose nodal set is asymptotic to $\mathcal{C}_{m,m}$.

\begin{enumerate}[label=(\roman*)]
    \item If the positive phase of $u$ is contained in $\mathcal{C}_{m,m}^+$, i.e., $\{u > 0\} \subseteq \{ |x| > |y| \}$, then
    \[
    u(x, y) \le u_s(x, y) \quad \text{for all } (x, y) \in \mathbb{R}^{2m}.
    \]

    \item If the positive phase of $u$ is contained in $\mathcal{C}_{m,m}^-$, i.e., $\{u > 0\} \subseteq \{ |x| < |y| \}$, then
    \[
    u(x, y) \le -u_s(x, y) \quad \text{for all } (x, y) \in \mathbb{R}^{2m}.
    \]
\end{enumerate}
\end{corollary}

\begin{remark}
As a consequence, one obtains a global pointwise bound for $u$ in terms of 1D solution $u_0$.
Indeed, it was proved in \cite{CT09} that the saddle-shaped solution $u_s$ satisfies
\[
|u_s(x, y)| \le \left| u_0\left( \mathrm{dist}\left((x, y), \mathcal{C}\right) \right) \right|
= \left| u_0\left( \frac{|x| - |y|}{\sqrt{2}} \right) \right| \quad \text{for all } (x, y) \in \mathbb{R}^{2m},
\]
where $u_0(z) = \tanh(z/\sqrt{2})$ denotes the 1D heteroclinic profile.
\end{remark}

\bibliographystyle{acm}
\bibliography{ref}

\end{document}